\documentclass{amsart}
\usepackage{amscd,amsthm}
\usepackage{latexsym}
\usepackage{capt-of}
\usepackage{graphicx}
\def\eps{\varepsilon}
\DeclareFontFamily{U}{wncy}{}
\DeclareFontShape{U}{wncy}{m}{n}{<->wncyr10}{}
\DeclareSymbolFont{mcy}{U}{wncy}{m}{n}
\DeclareMathSymbol{\Sha}{\mathord}{mcy}{"58}
\usepackage{tikz-cd}
\usepackage{amssymb,amsmath}
\usepackage{amsfonts}
\usepackage{tkz-graph}
\newcounter{ctfig}

\newcommand{\A}{\mathbb{A}}

\newcommand{\p}{{\mathfrak p}}
\newcommand{\q}{{\mathfrak q}}
\renewcommand{\P}{\mathbb{P}}

\newcommand{\codim}{\operatorname{codim}}
\newcommand{\cod}{\operatorname{codim}}

\newcommand{\Hom}{\mathop{\rm Hom}}
\newcommand{\cok}{\mathop{\rm cok}}

\theoremstyle{plain}
\newtheorem{thm}{Theorem}[section]

\newtheorem{lemma}[thm]{Lemma}
\newtheorem{prop}[thm]{Proposition}
\newtheorem{cor}[thm]{Corollary}
\newtheorem{crit}[thm]{Criterion}

\theoremstyle{definition}

\newtheorem{alg}[thm]{Algorithm}
\newtheorem{example}[thm]{Example}
\newtheorem{remark}[thm]{Remark}
\newtheorem{defn}[thm]{Definition}
\newtheorem{notation}[thm]{Notation}

\newtheorem{question}[thm]{Question}

\def\O{{\mathcal O}}

\def\m{{\mathfrak m}}
\newcommand{\Spec}{\mathop{\rm Spec}}

\newcommand{\im}{\mathop{\rm im}}

\newcommand*{\medcap}{\mathbin{\scalebox{1.2}{\ensuremath{\cap}}}}%
\newcommand*{\medcup}{\mathbin{\scalebox{1.2}{\ensuremath{\cup}}}}%

\begin{document}
\bibliographystyle{plain}
\bibstyle{plain}

\title[Crepant resolutions of double covers]{On the Cynk-Hulek criterion for crepant resolutions of double covers}

\author{Colin Ingalls}
\address[Colin Ingalls]{School of Mathematics and Statistics, 4302 Herzberg Laboratories,
  1125 Colonel By Drive, Carleton University, Ottawa, ON K1S 5B6, Canada}
\email{colin.ingalls@gmail.com}

\author{Adam Logan}
\address[Adam Logan]{The Tutte Institute for Mathematics and Computation,
P.O. Box 9703, Terminal, Ottawa, ON K1G 3Z4, Canada} % and Carleton
\address[Adam Logan]{School of Mathematics and Statistics, 4302 Herzberg Laboratories,
  1125 Colonel By Drive, Carleton University, Ottawa, ON K1S 5B6, Canada}
\email{adam.m.logan@gmail.com (corresponding author)}
%\author{Colin Ingalls \and Adam Logan}
%\thanks{$^*$ Corresponding author}
%\email{colin.ingalls@gmail.com, adam.m.logan@gmail.com}
%\address{The Tutte Institute for Mathematics and Computation,
%P.O. Box 9703, Terminal, Ottawa, ON K1G 3Z4, Canada} % and Carleton
%\address{School of Mathematics and Statistics, 4302 Herzberg Laboratories,
%  1125 Colonel By Drive, Carleton University, Ottawa, ON K1S 5B6, Canada}

\date{\today}

\begin{abstract}
  A collection $S = \{D_1,\ldots, D_n\}$ of divisors on a smooth
  variety $X$ is an {\em arrangement} if the intersection of every subset
  of $S$ is smooth.  We show that a double cover of $X$ ramified on
  an arrangement has a crepant resolution under additional hypotheses.
  Namely, we assume that all intersection components that change the
  canonical divisor when blown up satisfy are {\em splayed}, a property
  of the tangent spaces of the components first studied by Faber.
  This strengthens a result of Cynk and Hulek, which
  requires a stronger hypothesis on the intersection components.
  Further, we study the singular subscheme of the
  union of the divisors in $S$ and prove that it has a primary
  decomposition where the primary components are supported on exactly
  the subvarieties which are blown up in the course of constructing
  the crepant resolution of the double cover.
\end{abstract}

\maketitle

\section{Introduction}\label{sec:intro}

In this paper we will study {\em crepant} resolutions of singularities of
double covers.  These are resolutions that do not affect the canonical divisor.
Such resolutions have been studied for many reasons: for example, a large
number of interesting examples of rigid Calabi-Yau varieties have been
constructed as double covers of rational varieties as in \cite{meyer},
\cite{csv}.  In addition, crepant resolutions are very important in the
Mori program \cite{matsuki} and are related to derived equivalence as in
\cite{bkm}.

Recall that the singular locus of the double cover of a smooth variety
branched along a union of smooth divisors is supported above the intersection
of two or more of the divisors.  Thus we will need to understand how the
behaviour of intersections of divisors changes under blowing up.
In order to state our first main result, we introduce the concept of a
splayed set of divisors, which is due to Faber \cite{faber}.

\begin{defn}\label{def:arrangement}
  Given a variety $X$, let $S = \{D_1, \dots, D_n\}$ be a set of
  divisors on $X$ such that the intersection of every subset of $S$ is
  smooth.  Then $S$ %, or the $D_i$,
  constitutes an {\em arrangement} of divisors on $X$.
\end{defn}

\begin{defn}\label{def:splayed}\cite[Definition~2.3]{faber}
  Let $S$ be an arrangement on $X$ and let $p \in \medcap S$.  We say that
  $S$ is {\em splayed} at $p$ if we can write  $S = S_{1,p} \cup S_{2,p}$, where
  the $S_{i,p}$ are nonempty and
  $\medcap_{i \in S_{1,p}} T_p({D_i}) + \medcap_{j \in S_{2,p}} T_p({D_j}) = T_p(X)$.
  If $S$ is splayed at every point, it is {\em splayed}.  In particular,
  we say that $S$ is splayed if $\medcap S = \emptyset$.
\end{defn}

We now describe a sufficient condition for crepant resolutions to exist and
an algorithmic procedure for constructing them in cases satisfying our
condition.  Both of these are generalizations of the criterion
given by Cynk and Hulek \cite[Proposition 5.6]{cynk-hulek}.

\begin{defn}\label{def:admissible}
Let $C$ be an irreducible component of $\medcap_{i \in T} D_i$, where
$T \subset \{1,\dots,n\}$, and let $T_C \subset \{1,\dots,n\}$ be the
set of $i$ with $C \subseteq D_i$.  We then say that $C$ is
{\em admissible} if $|T_C| - 2\,\codim C \in \{-1,-2\}$.
\end{defn}

The justification for this definition is that the blowup of the base along
such a subvariety pulls back to a crepant blowup of the double cover,
as we explain in more detail in Proposition \ref{prop:admissible-is-crepant}.
We now state our main theorem.
\begin{thm} (Thm.~\ref{thm:resolve}) 
  Let $S = \{D_1, \dots, D_k\}$ be an arrangement on $X$.  Suppose that for
  all components $C$ of intersections $\medcap I$ for $I \subseteq S$, either
  $I$ is splayed along $C$ or the intersection is admissible.
  Then all double covers with branch locus $S$ admit a crepant resolution.
\end{thm}

We recall that such double covers exist if and only if the sum of the
Picard classes of the $D_i$ is a multiple of $2$.  In this case we 
construct the resolution in the following way.

\begin{alg}\cite[Proposition 5.6]{cynk-hulek}\label{alg:resolve}
  As long as there are
  non-splayed components of the intersection of the components of the branch
  locus, choose $C$ to be a minimal one and blow it up.  Replace $S$ by
  $\{\tilde D_1, \dots, \tilde D_k\}$ if the number of $D_i$ containing $C$
  is even, or by $\{\tilde D_1, \dots, \tilde D_k, E\}$ if it is odd, where
  $E$ is the exceptional divisor.
  When all of the components of the intersection are splayed, blow up all
  pairwise intersections of the components of the branch locus in arbitrary
  order.
  \end{alg}

We will show in Section \ref{sec:crepant-res} that this algorithm terminates
in a crepant resolution.
We now explain the connection of our work with \cite{cynk-hulek}.  
\begin{defn}\cite[below Lemma 5.5]{cynk-hulek}
Let $C$ be an irreducible component of $\medcap_{i \in S} D_i$, where
$S \subset \{1,\dots,n\}$, and let $S_C \subset \{1,\dots,n\}$ be the
set of $i$ with $C \subseteq D_i$.  We say that $C$ is
{\em near-pencil} if $\dim \medcap_{i \in S_C \setminus \{j\}} > \dim C$
for some $j \in S$.
\end{defn}

\begin{remark}\label{rem:splayed-1-2}
  According to this definition, $C$ is near-pencil if and only if the set of
  $D_i$ containing $C$ is splayed by $S, T$, where one of the sets $S, T$ is
  a singleton.  We also note that if the set of $D_i$ containing $C$ is splayed
  by $S,T$ where $\#S = 2$, then $C$ is still near-pencil.  Indeed, in local
  coordinates $S$ consists of the divisors $x = 0, y = 0$, while $T$ consists
  of divisors whose equations do not involve $x, y$.  Clearly we may move
  $y = 0$ from $S$ to $T$, obtaining a splaying with $\#S = 1$ and showing
  that $C$ is near-pencil.
\end{remark}

Cynk and Hulek give
the following criterion for a crepant resolution:
\begin{crit}\label{crit:cynk-hulek}
  The {\em Cynk-Hulek criterion} \cite[Proposition 5.6]{cynk-hulek}
  states that if every intersection is
admissible or near-pencil then there exists a crepant resolution of the double
cover $Y \to X$.
\end{crit}

\begin{remark} There is a slight error in the proof of
  \cite[Proposition 5.6]{cynk-hulek}: the two inequalities on lines
  13--14 of page 501 should be reversed and made strict.
  We thank Professors Cynk and Hulek for clarifying this point.
\end{remark}

In particular, Cynk and Hulek give an algorithm for constructing a resolution.
Our Algorithm \ref{alg:resolve} above follows it exactly, except that we
blow up non-splayed intersection components rather than non-near-pencil ones.
We do not need to blow up splayed intersection components.
In Section \ref{sec:primary-dec} we will explain
this striking fact by showing that the singular locus of $\medcup D_i$ admits
a primary decomposition with no components whose support is splayed.  To be
precise, we will prove the following:

\begin{thm} (Thm.~\ref{thm:pc})
  Let $D_i$ be an arrangement of divisors on $X$, and let $S$ be the singular
  subscheme of $\medcup_i D_i$.
  Let $E = \medcap_{i \in S} D_i$ be an intersection of
  some of the $D_i$ which is splayed and not of codimension $2$.
  Then the minimal primary decomposition of
  $S$ does not contain a component supported on $E$.
\end{thm}

Our theorem generalizes, not only the result of Cynk-Hulek, but also some
results such as \cite[Theorem 2.1]{cynk-szemberg} that have
significant arithmetical applications.  In fact, this was our motivation
for proving the results of this paper.
To conclude the introduction, we give an example of an arrangement that
can be given a crepant resolution by our procedure but not by that of Cynk and
Hulek.  Such an arrangement must have an intersection component that is splayed
but not admissible or near-pencil.
We believe that the existence of a crepant resolution in this case
was not previously known.

\begin{example}\label{ex:only-ours}
  Let $x_i$ for $0 \le i \le 6$ be coordinates on $\P^6$, and let
  $S = \{x_i = 0: 1 \le i \le 6\} \cup \{x_{2i-1}+x_{2i} = 0: 1 \le i \le 3$\}.
  It is clear that an intersection of a subset of $S$ fails to be near-pencil
  if and only if it is an intersection of one or more of the loci
  $x_{2i-1} = x_{2i} = 0$, and that an intersection of one or two of these loci
  is admissible while the intersection of all three is not.  On the other
  hand, the intersection of all three of these loci is splayed by the
  decomposition $(T,S \setminus T)$ where $T = \{x_1=0, x_2 = 0, x_1+x_2=0\}$.
  By adding the hyperplane $x_0 = 0$ to $S$ we obtain a set whose sum is even in
  the Picard group.
\end{example}

\section{Second-order jets and blowups}\label{sec:jets}

\subsection{Coordinate calculations}
Let $V$ be a variety and $p \in V$ a point.
Let $T^i_p V= \Hom(\O_{V,p},k[\varepsilon]/(\varepsilon^{i+1}))$ be
the set of $k$-algebra homomorphisms, which we interpret as the space
of $i$-jets of $V$ at $p$.  For a map $Y \stackrel{f}{\to} X$ taking
$q \in Y$ to $p \in X$,
there is a natural commutative diagram
$$\begin{CD}
  T^{i+1}_q Y & @>{df}>> & T^{i+1}_p X \\
  @V{\pi^i_q}VV & & @V{\pi^i_p}VV \\
  T^{i}_q Y & @>{df}>> & T^{i}_p X
\end{CD}$$
with row maps induced by composing a map $w : \O_{Y,q} \to
k[\eps]/(\eps^j)$ with the map $\O_{X,p} \to \O_{Y,q}$ and
column maps induced by composing with the natural map
$k[\eps]/(\eps^{i+1}) \to k[\eps]/(\eps^i)$.

\begin{lemma}\label{lemma:phs} Let $R$ be a regular local $k$-algebra with
  residue field $k$ and let $n$ be a positive integer.  Define
  $T^nR = \Hom(R,k[\eps]/(\eps^{n+1}))$ to be the set of $k$-algebra
  homomorphisms.  There is a natural exact sequence of pointed sets
  $$ 0 \to T^1R \stackrel{\varphi^n}{\to} T^nR \stackrel{\pi^n}{\to}
  T^{n-1}R \to 0$$ giving $T^nR$ the structure of a $T^1R$-torsor with
  base $T^{n-1}R$.
\end{lemma}
\begin{proof}
  There is a natural exact sequence
  $$ 0 \to \eps^nk[\eps]/(\eps^{n+1}) \to k[\eps]/(\eps^{n+1}) \to
  k[\eps]/(\eps^{n}) \to 0.$$
  We identify $\eps^nk[\eps]/(\eps^{n+1})$ 
  with $\eps k[\eps]/(\eps^2)$ by multiplication by $\eps^{n-1}$.
  We apply the left
  exact functor $\Hom(R,-)$ to the exact sequence.  The resulting
  sequence is exact since $R$ is regular.
\end{proof}

\begin{defn}\label{def:p-sub-u}
  Let $u \in T^n R$.  Then the set $T_u^n(R) = (\pi^n)^{-1}(\pi^n(u))$ is a
  principal homogeneous space for the vector space $T^1 R$.  Choosing
  $u$ to be the origin, we trivialize the principal homogeneous space,
  making $T_u^n(R)$ into a vector space with an isomorphism
  $\tau_u: T^1 R \stackrel{\cong}{\to} T_u^n(R)$.

  If $R = \O_{V,p}$ is
  the local ring of a variety $V$ at a point $p$, then we write
  $T^n_{p,u}(V)$ for $T_u^n(R)$.
  When $p$ is a smooth point of $V$, 
  we denote the map $\pi^n$ of Lemma \ref{lemma:phs} for
  the local ring $\O_{V,p}$ by $\pi^n_p$.
\end{defn}

As in Section \ref{sec:intro}, let $X$ be a smooth variety.  Let $C
\subset X$ be a smooth subvariety and let $f:\widetilde{X} \to X$ be
the blowup of $X$ along $C$ with exceptional divisor $E$.  Let $q \in E$
and let $p = f(q)$.  Our goal is to describe an injective map from
a jet space at $q$ to a jet space at $p$.  Just as the choice of a
point $q$ in $E$ determines a tangent direction at $p$, first-order
jets at $q$ give information about second-order jets at $p$ that are
in the direction of $q$ at first order.  We will apply this to obtain
a bound on the dimensions of tangent spaces of intersections of divisors
and of their proper transforms.

Pick \'etale local coordinates $x_0,\ldots,x_n,t_1,\ldots,t_m$ for $p$
in $X$ in such a way that $C = V(x_0,x_1,\ldots,x_n)$.
There exist \'etale covers of
open neighbourhoods of $p,q$ in which we can write $f:\widetilde{X} \to X$ as
$$(x_0,x_1,\ldots,x_n,t_1,\ldots,t_m) \mapsto (x_0,x_1/x_0,\ldots,x_n/x_0,t_1,\ldots,t_m)$$ and the points
are described by $$p = V(x_0,\ldots,x_n,t_1,\ldots,t_m), \quad \quad q = V(x_0,x_1/x_0 - \alpha_1,\ldots,x_n/x_0 - \alpha_n,t_1,\ldots,t_m)$$ for fixed $\alpha_i$ in $k$.  Note that
$q \in E = V(x_0)$.

Let $w$ be an arbitrary element of $T^2_q \widetilde{X}$ defined by
\begin{equation}\label{eqn:w}
  \begin{aligned}
    w(x_0) &= \beta_0 \eps + \gamma_0 \eps^2, \\
    w(x_1/x_0 - \alpha_1) &= \beta_1 \eps + \gamma_1 \eps^2, \dots, \\
    w(x_n/x_0 - \alpha_n) &= \beta_n \eps + \gamma_n \eps^2, \\
    w(t_1) &= \eta_1 \eps + \theta_1 \eps^2, \dots, \\
    w(t_m) &= \eta_m \eps + \theta_m \eps^2.
  \end{aligned}
\end{equation}
Then $df(w)(x_i) = w(x_0(x_i/x_0 - \alpha_i + \alpha_i))$ for $i > 0$,
while $df(w)(x_0) = w(x_0)$ and $df(w)(t_j) = w(t_j)$, so
$df(w) \in T^2_p X$ is described by
\begin{equation}\label{eqn:f-of-w}
  \begin{aligned}
    df(w)(x_0) &= \beta_0 \eps + \gamma_0 \eps^2,\\
    df(w)(x_1) &= (\alpha_1 + \beta_1 \eps + \gamma_1 \eps^2)(\beta_0 \eps + \gamma_0 \eps^2), \dots,\\
    df(w)(x_n) &= (\alpha_n + \beta_n \eps + \gamma_n \eps^2)(\beta_0 \eps + \gamma_0 \eps^2), \\
    df(w)(t_1) &= \eta_1 \eps + \theta_1 \eps^2, \dots, \\
  df(w)(t_m) &= \eta_m \eps + \theta_m \eps^2.
  \end{aligned}
\end{equation}
Note that the point $q  = V(x_0,x_1/x_0 - \alpha_1,\ldots,x_n/x_0 - \alpha_n,
t_1,\ldots,t_m)$ together with a fixed $\beta_0 \ne 0 \in k$ determines a
nonzero normal vector $n_q \in N_p X/C$ given by
\begin{equation}\label{eqn:vq}
  \begin{aligned}
    n_q(x_0) & = & \beta_0\varepsilon\\
    n_q(x_1) & = & (\alpha_1 +\beta_1 \eps)(  \beta_0\varepsilon) =\alpha_1\beta_0\varepsilon \\
    \vdots \\
    n_q(x_n) & = & (\alpha_n +\beta_n \eps)( \beta_0\varepsilon) =\alpha_n\beta_0\varepsilon       \\          
    % v_q(t_1) & = & \eta_1 \eps \\
    %   \vdots \\
    %  v_q(t_m) & = & \eta_m \eps
  \end{aligned}
\end{equation}

\begin{defn}\label{def:vprime-q}
  Let $v_q$ be the lift of $n_q$ to $T_p^1(X)$ that satisfies $v_q(t_i) = 0$
  for $1 \le i \le m$.
\end{defn}

\begin{lemma}\label{lemma:t-surj}
  The map $T_q^1(\tilde X) \to T_p^1(C) + \langle v_q \rangle$ is surjective.
\end{lemma}

\begin{proof}
  The subspace $T_p^1(C) + \langle v_q \rangle$ does not depend on the
  choice of $\beta_0$, because multiplying $\beta_0$ by $c \ne 0$
  replaces $v_q$ by $cv_q$.
  The image of the map contains $T_p^1(C)$, because the
  map (\ref{eqn:f-of-w}) takes the $t_i$ to the space spanned by the
  $\eta_i \eps$.  Let $w$ be a vector such that $df(w)(x_0) = \beta_0 \eps$
  and $df(w)(x_1), \dots, df(w)(x_n), df(w)(t_1), \dots, df(w)(t_m)$ are
  multiples of $\eps^2$.  Then
  $df(w) \in (v_q + T_p^1(C)) \setminus T_p^1(C)$.  The result follows.
\end{proof}

We choose $u_q \in T^2_q\widetilde{X}$ such that
%$df \pi^1_q(u_q) = \pi^1_p df (u_q) = v_q$.
$df(u_q) = v_q$.
% Let $T^2_{p,v_q}X \subset T^2_pX = \eps \cdot (\pi^1_p)^{-1}(v_q).$
%be the inverse image of $v_q$ by $\pi_p^1$.
Lemma~\ref{lemma:phs} makes
$T^2_{p,v_q}(X)$ into a principal homogeneous space over $T^1_p(X)$; choosing
$df(u_q)$ as origin, we make $T^2_{p,v_q}(X)$ into a vector space with an
isomorphism $\tau_{v_q}: T^1_p(X) \to T^2_{p,v_q}(X)$. We similarly define $T^2_{q,u_q}$.
This amounts to an identification between the first-order
tangent space and the space of second-order tangent vectors that restrict
to a given first-order tangent vector.  Such an identification depends
crucially on the smoothness of $X$ at $p$, or, in the algebraic formulation
of Lemma \ref{lemma:phs}, on the regularity of~$R$.

Also, $(df)^{-1}(v_q)$ is a principal
homogeneous space for $T^1_qF$.
%\footnote{Colin: should we change all the $df$s to $f_*$?}
Choosing $\pi^1_q(u_q)$ as origin, we obtain an isomorphism
$\tau_{u_q}: T^1_q F \to (df)^{-1}(v_q)$.
%\footnote{Colin: we should perhaps make a
%  commutative diagram to relate this to $f$ and $\varphi^2$.}
%\begin{remark}\label{rem:t1-qvq}
  %Let $T^1_{q,v_q}(\widetilde X) = (df)^{-1}(v_q)$.\footnote{Adam:
  %We already have a definition for $T^i_{p,v}$, so this is not ideal.}
  %Note that $T^1_{q,v_q}(\widetilde X) = (df)^{-1}(v_q)$.
%  When we think of $T^1_{q,v_q}(\widetilde X)$
%  as a set or a principal homogeneous space, we will write it as such;
%  however, when we think of it as a vector space in its own right,
%  we will denote it by $T^1_{q,v_q}(\widetilde X) - \pi^1_q(u_q)$.
%\end{remark}
%In other words, we introduce a vector space
%structure on $(df)^{-1}(v_q)=T^1_{q,v_q}\widetilde{X} - \pi^1_q(u_q)$.
%We define a linear map from $T^1_qF$ to $T^2_{p,v_q}(X)$ by composing
%$df$ with $\tau_{v_q}$.
% The map $\varphi^2$ can be thought of as multiplication by $\eps$
%Note that for $g \in T^2_{p,v_q}$ the map $g-f(u_q)$ is zero on $k[\eps]/(\eps^2)$
%and so determines a map $g-f(u_q) : \O_{\widetilde{X},q} \to \eps k[\eps]/(\eps^2)$.
%Since $\widetilde{X}$ is smooth there are no obstructions to lifting $f(u_q)$
% and so we obtain a $n+m+1$ dimensional vector space.
We choose $\lambda : T^1_q \widetilde{X} \to T^2_q\widetilde{X}$ to be the map
such that
$\lambda(w)(x_0) = \beta_0\eps,\lambda(w)(x_i/x_0-\alpha_i) = \beta_i\eps,
\lambda(w)(t_i) = \eta_i\eps$, splitting the natural projection
$T^2_q\widetilde{X} \to  T^1_q \widetilde{X}$.  
We summarize the definitions above in the diagram below.
\begin{equation}\label{diag:rho}
  \begin{tikzcd}
  T^1_qF \arrow[r,"\tau_{u_q}"] \arrow[rd,"\rho"] &  (df)^{-1}(v_q)  \arrow[r,"{\lambda}"]  &  T^2_{q,u_q} \widetilde{X}   \arrow[r] \arrow[d] &   T^2_q\widetilde{X} \arrow[d,"df"]\\
 &  T^1_p(X) \arrow[r,"\tau_{v_q}"] &  T^2_{p,v_q} X   \arrow[r] &  T^2_p X 
\end{tikzcd}\end{equation}
The map $\rho$ is defined so as to make the diagram commute.
That is, we define
$\rho := \tau_{v_q}^{-1}\circ df\circ \lambda \circ \tau_{u_q}: T^1_qF \to T^1_pX$.

\begin{prop}\label{prop:T1toT2Injective}
  The subspaces $(df)(T^1_q\widetilde{X})$ and $\rho(T^1_qF)$ are complementary
  in $T^1_p X$; hence $\rho$ induces an isomorphism
  $T^1_qF \to \cok (df :T^1_q\widetilde{X} \to T^1_pX)$.
%  The map $T^1_{q,v_q}(\widetilde{X}) - \pi^1_q(u_q) \to T^2_{p,v_q}(X) - v_q$
%  is injective.
%  Further, the subspace spanned by the image of $T_p(C) + \langle v_q \rangle$
%  is a complement of the image, which is naturally isomorphic to $T^1_q(F)$.
  \end{prop}
\begin{proof}
This follows from an explicit computation of the map $\rho$.
The nonzero tangent vector $u_q$ determines $v_q$, and so a scalar
$\beta_0 \neq 0$.
Let $w \in T^1_{q,v_q}(\widetilde X)$.
In equation~(\ref{eqn:w}), we set $w(t_j) = \eps^2 = 0$ and observe that
$w(x_0) = \beta_0$ and $
w(x_i/x_0 - \alpha_i) = \beta_i \eps$ for $1 \le i \le n$.
Note that $\eta_1 = \cdots = \eta_m = 0$
for all vectors in $T^1_{q,v_q}(\widetilde{X})$ since
these are directions contained in $F$.

Similarly, equation~(\ref{eqn:f-of-w})
gives $df(w)(x_0) = \beta_0 \eps$,
$df(w)(x_i) = (\alpha_i + \beta_i \eps)(\beta_0 \eps) = \alpha_i \beta_0 \eps + \beta_i \beta_0 \eps^2)$ for $1 \le i \le n$, and
again $df(w)(t_1) = \dots = df(w)(t_m) = 0$. %\footnote{Adam, Colin:
Hence the linear
map $T^1_qF \to T^1_p X$
% $T^1_{q,v_q}(\widetilde{X}) - \pi^1_q(u_q) \to T^2_{p,v_q}(X) - v_q$
is given in coordinates by
\begin{equation}\label{eqn:localMap}
  \eps\begin{pmatrix} \beta_1 \\
  \vdots \\ \beta_n \end{pmatrix}  \mapsto \eps^2\begin{pmatrix}0 \\ \beta_1\beta_0 \\
  \vdots \\ \beta_n\beta_0 \\ 0 \\ \vdots \\ 0 \end{pmatrix}.
\end{equation}
\end{proof}

\subsection{Maps on tangent spaces}
We will apply the results of the previous section to prove a
proposition that will be a key step
toward Theorem~\ref{thm:resolve}, our main result.  We will establish the
following:
\begin{prop}\label{prop:blowupIsArrangement} Let $\{D_i\}_{i \in S}$ be an arrangement in $X$,
  and let $\widetilde{X}$ be the blowup of $X$ along a component $C$ of
  the intersection of some subset of the $S$, with exceptional divisor $E$.
  Then $\{ \widetilde{D}_i\}_{i \in S} \cup \{E\}$ is an arrangement in $\widetilde{X}$.
\end{prop}
We will do so by proving a sequence of lemmas and propositions.
Let $I \subseteq S$ be an
arbitrary subset of $S$ and let $C'$ be a component of $\medcap_{i \in I} D_i$.
\begin{lemma}\label{lemma:CdoesnotcontainCprime}
  If $C \supseteq C'$ then $E \cap \medcap_{i \in I} \widetilde{D}_i = \emptyset$.
  \end{lemma}
  \begin{proof}
    We have a sequence of natural maps
    $$ T_pC' \to T_pX \to N_{p,X/C}.$$
    The projectivization of the image of the composite map
    is naturally identified with the inclusion $E \cap \medcap_{i \in I} \widetilde{D}_i \subset \P(N_{X/C})$.
    Since $C' \subseteq C$, the map $T_pC' \to T_pX$ factors through $T_pC$
    and so the image is zero.
  \end{proof}
Our next task is to understand
$T_q \cap \widetilde{D}_i$ by means of the filtration
$$T_qF \subseteq T_qE \subset T_q\widetilde{X}.$$
We will do this in three steps.  In the first step,
we use the results of the last section to prove the following statement:
\begin{prop} With notation as above,
  $$ \dim \left( T_qF \cap \medcap_{i \in I} T_q \widetilde{D}_i \right)
  =  \dim T_p C' - \dim T_p(C \cap C') -1.$$
\end{prop}
\begin{proof} We consider the injective map $T^1_{q,v_q}(\widetilde{X}) - \pi^1_q(u_q) \to T^2_{p,v_q}(X) - v_q$ from Proposition~\ref{prop:T1toT2Injective}.  Let $c$ be the codimension of $C \cap C'$ in $C'$.
  Since $C$, $C'$, and $C \cap C'$ are smooth, we may choose local coordinates
  at $q$ so that
  $C \cap C'$ is locally defined by $V(x_0,\ldots,x_{c-1})$ in $C'$
  and $C' = V(x_c,\ldots,x_m,t_1,\ldots,t_j)$.
    So we obtain $c$ linearly independent
  conditions in $T^2_{p,v_q}(X) - v_q$. The map~(\ref{eqn:localMap}) from 
  $T_q^1F$ is injective, so
  we obtain $c-1$ conditions in $T_q^1F$, since $x_0 \mapsto 0$.
  The fact that $\medcap_{i \in I}\widetilde{D}_i$ surjects onto $C'$ implies
  that $c \geq \cod \medcap_{i \in I} \widetilde{D}_i$.
  Furthermore $c+1 \geq \cod E \cap \medcap_{i \in I} \widetilde{D}_i$.
  For a particular $\tilde{D}_i$ we can choose coordinates so
  $\tilde{D}_i = V(x_{j_i}+\sum_{k=j+1}^m \alpha_k t_k)$.
  We map $T^1F \to TC'/T(C \cap C')$.
  Note that $T(C \cap C')$ is spanned by
  $\partial t_{j+1},\ldots,\partial t_m$, and $TC'$
  is spanned by
  $\partial x_0,\ldots,x_{i-1},\partial t_{j+1},\ldots,\partial t_m$.
  We claim that~(\ref{eqn:localMap}) restricts to an isomorphism
  $T^1F \cap \medcap_i T\widetilde{D}_i$
  to the span of $x_1,\ldots,x_{i-1}$ in $T^2X - v_q$.
  Note that since $\partial t_j = 0$  in $TC'/T(C \cap C')$,
  we only need to consider the $x_i$ term in the local equation of
  $\tilde{D}_i$.
  Since the $\medcap_i D_i = C'$ we obtain all $\partial x_i$ in the image of
  $T_q\tilde{D}_i$.
  \end{proof}
In the second step we analyze $\dim T_q(E \cap \medcap_{i} \widetilde{D}_i)$.

\begin{lemma}\label{lem:tangent-space-meet} Let $S$ be a set of subvarieties
  of a fixed variety.  For all $p \in \cap S$
  we have $T_p(\cap S) = \medcap_{s \in S} T_p(S)$.
\end{lemma}

\begin{proof} A tangent vector is contained in $\medcap S$ if and only if it is
  contained in every element of $S$.  Algebraically, consider a map from the
  local ring at $p$ to $\Spec(k[\eps]/(\eps^2))$; it factors through the
  coordinate ring of each $D_i$ if and only if the kernel contains the ideal of
  $D_i$.  So this happens for all $D_i$ if and only if the kernel contains
  the ideal of every $D_i$, if and only if it contains their sum, if and only
  if it contains the ideal of the intersection.
\end{proof}

\begin{prop}\label{prop:step2}
$\dim T_q(E \cap \medcap_{i} \widetilde{D}_i) = 
\dim (C \cap C') + \dim T_q(F\cap \medcap_i \widetilde{D}_i).$
\end{prop}

\begin{proof}
  Consider the following commutative diagram.
  $$\begin{CD}
  0 & @>>> & T_q (F \cap \medcap_i \widetilde{D}_i) & @>>> & T_q(E \cap \medcap_i \widetilde{D}_i)
  & @>{df}>> & T_p(C \cap C') & @>>> & 0 \\
  &  & & & @VVV & &  @VVV & &   @VVV \\
  0 & @>>> &  T_q F &   @>>> &  T_qE &   @>{df}>> &  T_pC & @>>> & 0
  \end{CD}$$
  Note that the bottom row is exact since $E \to C$ is a fibre bundle with fibre
  $F$.  Since $f$ maps $E \to C$ and $\widetilde{D_i} \to D_i$ we see that $df$ maps 
$T_q(E \cap \medcap_i \widetilde{D}_i)$ to  $T_p(C \cap C')$.  
  We will show that the top row is exact as well.  The map
  $$ T_q (F \cap \medcap_i \widetilde{D}_i) \to T_q(E \cap \medcap_i \widetilde{D}_i)$$
  is a restriction of an injective map, and so clearly injective.
  In addition, $F \subseteq E$ implies that $T_q(F) \subseteq T_q(E)$.
  Lemma \ref{lem:tangent-space-meet} shows that $T_q(\medcap_i V_i) = \medcap_i T_q(V_i)$.
  Thus
  $$T_q(F \cap \medcap_i \widetilde{D}_i) = T_q F \cap T_q(E \cap \medcap_i \widetilde{D}_i),$$ and restriction is left exact so the sequence is exact in the middle.
  Finally, we need to show surjectivity of the last map.

  We have a commutative diagram
  $$ \begin{CD} T_q (E \cap \widetilde{C'}) @>>> T_q (E \cap \medcap_{i \in I} \widetilde{D}_i) \\
    @VVV @VVV \\
    T_p (C \cap C') @= T_p (C \cap C').\end{CD} $$ 
  The top map is injective and the left vertical map is surjective, so the right vertical map is also surjective.
\end{proof}

In the third and final step, we look at $\dim T_q \medcap_{i \in I} \widetilde{D}_i$.
\begin{prop}\label{prop:step3}
  $$\dim T_q \bigcap_{i \in I} \widetilde{D}_i =  \dim T_q(E \cap \bigcap_{i \in I} \widetilde{D}_i) +1.$$
  \end{prop}

\begin{proof}
  By Lemma~\ref{lemma:CdoesnotcontainCprime}, we can assume that
  $C \not\subseteq C'$.  So we can find
  $v_{c'} \in T^2_{p,v_q}(C') \setminus T^2_{p,v_q}(C)$ which we lift to
  $\widetilde{v}_{c'} \in T^1_q\medcap_{i \in I}\widetilde{D}_i \setminus T_q^1 E$.
\end{proof}

\begin{prop}\label{prop:tangent-map}%\footnote{Colin: ideally this would come
    %right after Proposition \ref{prop:T1toT2Injective}.}
  Consider the commutative diagram
  $$
  \begin{CD}
    T_q(F) @>>> T_q(E) @>>> T_q({\tilde X}) @= T_q({\tilde X})\\
    @VVV @VVV @V{df}VV @V{df}VV \\
    0 @>>> T_p(C) @>>> T_p(C) + \langle v_q \rangle @>>> T_p(X)
  \end{CD}
  $$
  in which the vertical maps are induced by the blowup $f: \tilde X \to X$ 
  (recall that $n_q$ is a vector in $N_{C/X}(p)$ corresponding to $q$ and that
  $v_q$ is a fixed lift of $n_q$ to $T_p(X)$).  The
  first three vertical maps are surjective with kernel $T_q(F)$.
  There is a natural isomorphism $r:T_q(F) \stackrel{\sim}{\to} \cok df.$
  % \footnote{Colin: it would be much better
   % to describe this as a complement, since the natural map is from $T_q(F)$
   % to $T_p(X)$, not the other way around} and our results on second-order jets
 % gives the isomorphism.\footnote{Adam: the last bit needs to be more precise.}
\end{prop}
\begin{proof}
  The first two vertical maps are clearly surjective, and the kernel
  of the second map is $T_q(F)$ since the map $E \to C$ is a
  projective space bundle with fibre $F$. The image $\im(df)$ of
  $T_q({\tilde X})$ inside $T_p(X)$ contains $T_p(C)$ and is
  not contained in $T_p(C)$ by Lemma \ref{lemma:t-surj},
  so the image is $T_p(C) + \langle v_q \rangle$.
  So the kernel of $df$ is equal to $T_q(F)$ since
  $\codim_{T_q({\tilde X)}}T_q(E) = \codim_{T_p(C) + \langle v_q \rangle} T_p(C)$.

  We identify $T_q(F)$ as a complement to $T_p(C) +\langle v_q \rangle$
  as in Proposition \ref{prop:T1toT2Injective}.
  Hence $\rho : T_q(F) \to T_p(X)$ induces an isomorphism $r:T_q(F) \to \cok df$ by composing with the
  quotient map $T_p (X) /(T_p(C)+\langle v_q \rangle)$ as
  described in equation~(\ref{diag:rho}). %\footnote{Colin, Adam: finish this
   % proof}
\end{proof}

\begin{cor}\label{cor:fibral-dirs}
  Let $D$ be a divisor on $X$ containing $p$ and such that $\tilde D$
  contains $q$ but $T_q(\tilde D) \not \supseteq T_q(F)$.
  Restriction of the map $r$ above induces an isomorphism  $r:T_q(F) \cap T_q(\tilde D) \to N_p(C/D)/\langle v_q \rangle$.
\end{cor}

\begin{proof}
  This follows by replacing $X$ by $D$ in Proposition \ref{prop:tangent-map}
  and noting that $T_q(F) \cap T_q(\tilde D)$ can be identified with 
  $T_q(F_D)$, where $F_D$ is the fibre above $p$ in the blowup of $D$ along
  $C \cap D$.
\end{proof}

\section{Crepant resolutions of double covers}\label{sec:crepant-res}
Recall that our goal is to prove a strengthened version of the Cynk-Hulek
criterion \ref{crit:cynk-hulek}
for the existence of crepant resolutions of double covers.  Our form of the
criterion replaces their condition of an intersection
being near-pencil with the weaker notion of splayedness, already mentioned
in the introduction and recalled below.
Prior to stating the theorem, we recall some notation.
  
\begin{defn}\label{def:arrangement-2} (same as Definition \ref{def:arrangement})
  Let $X$ be a projective variety and $S = \{D_1, \dots, D_n\}$
  a set of divisors on $X$.  Then $S$ is an {\em arrangement} if the
  scheme-theoretic intersection of every subset of $S$ is smooth.
\end{defn}

In particular $X$ and the $D_i$ are all smooth (the former because it is
the intersection of the empty subset of $S$).  We do not assume that the
$D_i$ are dimensionally transverse.  Note that, for all subsets $T \subseteq S$,
the connected and irreducible components of $\medcap T$ coincide, because a
connected but reducible component would be singular along a nonempty
intersection of some of the irreducible components.
We now review a standard condition for a blowup of the base not to affect
the canonical divisor of the double cover.

\begin{defn}\label{def:admissible-again}
  (same as Definition \ref{def:admissible})
  Let $C$ be an irreducible component of $\medcap_{i \in T} D_i$, where
  $T \subset \{1,\dots,n\}$, and let $S_C \subset \{1,\dots,n\}$ be the
  set of $i$ with $C \subseteq D_i$.  We then say that $C$ is
  {\em admissible} if $|S_C| - 2\,\codim C \in \{-1,-2\}$.
\end{defn}

This definition is justified by the following proposition, in view of
which blowing up an admissible component of the intersection does not
interfere with constructing a crepant resolution.
\begin{prop}\label{prop:admissible-is-crepant}
  Let $X$ be a variety with an arrangement of divisors $\{D_i\}$, let $Y$
  be a double cover branched on $\medcup_i D_i$, and let $C$ be an admissible
  component of $\medcap_{i \in S_C} D_i$.  Let $\phi: \tilde X \to X$ be the
  blowup of $X$ along $C$ and let $E$ be the exceptional divisor.  Then
  the branch locus of $Y \times_X \tilde X \to \tilde X$ is
  $\medcup_i \tilde D_i$ if $\#S_C$ is even or $\medcup_i \tilde D_i \cup E$
  if $\#S_C$ is odd.  Further, the canonical divisor of $Y \times_X \tilde X$
  is the pullback of that of $Y$.
\end{prop}

\begin{proof}
  Let $c = \codim C$.
  Since $\phi$ has degree $1$ it cannot introduce or remove any branch
  components other than the exceptional divisor.  The class of $\tilde D_i$ is
  $\phi^* [D_i]$ if $i \notin S_C$ or $\phi^* [D_i] - [E]$ if $i \in S_C$.
  Since $[E]$ is not divisible by $2$, it becomes part of the branch locus if
  and only if $\#S_C$ is odd.
  
  The canonical class of $\tilde X$ is $\phi^*(K_X) + (c-1) E$.   Thus
  $2K_{\tilde X} + \sum_i [\tilde D_i] = \phi^*(K_X)$ if $c = |S_C+2|/2$,
  while $2K_{\tilde X} + \sum_i [\tilde D_i] + [E] = \phi^*(K_X)$ if
  $c = |S_C+1|/2$, as claimed.
\end{proof}

We now recall the definition that is key for the statement and proof of our
main theorem.
\begin{defn}\label{def:splayed-2}\cite[Defn.~2.3]{faber}
  With $X, S$ as above, we say that $S$ is {\em splayed} at a
  point $p \in \medcap S$ if we can write  $S = S_{1,p} \cup S_{2,p}$, where
  the $S_{i,p}$ are nonempty and
  $\medcap_{i \in S_{1,p}} T_p({D_i}) + \medcap_{j \in S_{2,p}} T_p({D_j}) = T_p(X)$.
  If $S$ is splayed at every point, it is {\em splayed}.  In particular,
  we say that $S$ is splayed if $\medcap S = \emptyset$.
\end{defn}
Informally, this means that we can set local coordinates and partition the
divisors into two sets such that the linear terms in the equations of
divisors in one set do not involve any variable mentioned in the other.

Our goal is to prove the following theorem:

\begin{thm}\label{thm:resolve} Let $S = \{D_1, \dots, D_k\}$ be an arrangement
  of divisors on a variety $X$.
  %, i.e., the $D_i$ are irreducible and all intersections are nonsingular.
  Suppose that, for all components
  $C$ of intersections $\medcap I$ for $I\subseteq S$,
  either $I$ is splayed along $C$ or the intersection is admissible.
  Then all double
  covers with branch locus $S$ admit a crepant resolution.
\end{thm}

We remind the reader that if the sum of the Picard classes of all the $D_i$ is
not divisible by $2$, then there are no such double covers and the statement
holds vacuously; however, we will not assume this divisibility in our
argument.
The main ingredient in our proof is the following result.

\begin{thm}\label{thm:main}
  Let $S$ be an arrangement of divisors on $X$.  Let $C$ be a minimal
  non-splayed component of the intersection of a subset of $S$; in other
  words, every component of the intersection
  of $C$ with additional divisors taken from $S$ is splayed.
  Let $\pi: \tilde X \to X$ be the blowup along $C$, let $\tilde S$ be the
  set of strict transforms of elements of $S$, and let $E$ be the exceptional
  divisor.  Then:
  \begin{enumerate}
  \item There are fewer
    non-splayed components of intersections of subsets of $\tilde S$
    than of intersections of subsets of $S$.
  \item All components of intersections of subsets of $\tilde S$ with $E$
    are splayed.
  \item $\tilde S \cup \{E\}$ is an arrangement.
  \end{enumerate}
\end{thm}

The theorem above and the following lemma imply Criterion \ref{crit:cynk-hulek}.

\begin{lemma}\label{lem:near-pencil-is-splayed}
  Let $S$ be an arrangement of divisors on $X$. Let $I \subseteq S$ and 
  let $C$ be a component of the intersection $\medcap I$.
  Let $S_C = \{ D_i \in S \,|\, C\subset D_i \}$.
  If $C$ is near-pencil then $\medcup S_C$ is splayed along $C$.
\end{lemma}

\begin{proof}  
  Recall that an intersection component $C$ of $\medcap I$ is near-pencil
  if there is $D_j \in I$ such that the component of $\medcap_{D \in I, D \ne D_j} D$
  containing $C$ has dimension larger than $\dim C$; this implies that $I$ is
  splayed along $C$ by $((I \setminus \{D_j\}),\{D_j\})$.
\end{proof}

The proof of Theorem~\ref{thm:main} will proceed
by a sequence of lemmas, propositions, and corollaries.
We will prove (1) and (2) in
Proposition~\ref{prop:blowup-still-splayed} and (3) immediately after.

\begin{prop}\label{prop:constant-splaying}
  Suppose that $S$ is an arrangement and that $S$ is splayed at
  $p$ by $(S_1,S_2)$.  Then $S$ is splayed by $(S_1,S_2)$ at every point of
  the component of $\medcap S$ containing $p$.
\end{prop}

\begin{proof} Points on the same component of $\medcap S$ are also on the same
  component of $\medcap_{i \in S_{1,p}} D_i$.  Since the intersection is smooth,
  the dimension of $\medcap_{i \in S_{1,p}} T(D_i)$ is constant on each component, and
  likewise for $\medcap_{j \in S_{2,p}} T(D_j)$ and $\medcap_{k \in S} T(D_k)$.
  Applying the equality $\dim (V+W) = \dim V + \dim W - \dim V \cap W$
  completes the proof.
\end{proof}

Thus we will sometimes abuse language by saying that $S$ is splayed by
$(S_1,S_2)$ on the component of $\medcap S$, rather than at the point $p$.

\begin{lemma}\label{lem:subset-splayed}
  Let $T \subset U \subseteq S$ be subsets of an arrangement on a variety $X$,
  let $C$ be
  a component of $\medcap U$, let $p$ be a point of $C$, and let $U = (U_1,U_2)$
  be a splaying at $p$.  Suppose that $T \cap U_i$ is nonempty for $i = 1, 2$.
  Then $T$ is splayed at $p$ by $(T \cap U_1, T \cap U_2)$.
\end{lemma}

\begin{proof}
  Let $R_i = \medcap_{D \in (T \cap U_i)} T_p(D)$ and let
  $S_i = \medcap_{D \in U_i} T_p(D)$ for $i = 1, 2$.
  The hypothesis that $U$ is splayed at $p$ by $(U_1,U_2)$ states that
  $S_1 + S_2 = T_p(X)$.  It follows that $R_1 + R_2 = T_p(X)$, since
  $R_i \supseteq S_i$.  Since $T \cap U_i$ is nonempty for $i = 1, 2$, that
  means that $(T \cap U_1,T \cap U_2)$ is a splaying.
\end{proof}

\begin{prop}\label{prop:one-sided}
  Suppose that $S$ is an arrangement of divisors on $X$
  and $T$ is a subset of $S$ whose
  intersection has a connected component $C$ not contained in any elements
  of $S$ not in $T$ and minimal among non-splayed components of intersections
  of subsets of $S$.  Let $U$ properly contain $T$.  Then
  % either $\medcap U$ is disjoint from $C$ or
  $U$ is splayed along $C \cap \medcap U$
  as $(T,U \setminus T)$.
\end{prop}

\begin{remark}\label{remark:why-so-hard}
  This somewhat baroque wording is necessary for the result to apply in
  cases where
  the intersection of $T$ has two components, one where it is splayed and one
  where it is not, and we have an intersection of the component where it is
  splayed with certain divisors in a bad configuration.  For example, if
  we blow up $x = y = 0$ in $\P^3$, the intersection of the strict transforms of
  $z = 0$ and $y+z = 0$ is not smooth.  Now $\{x=0,y=0\}$ is splayed, but
  it could be that one component of intersection is locally
  like this, but there is another component of the intersection disjoint from
  that one and minimal non-splayed.
\end{remark}

\begin{proof} If $C \cap \medcap U$ is empty, the statement holds vacuously.
  Otherwise, let $p \in C \cap \medcap U$,
  and let $U = U_1 \cap U_2$ be a splaying at $p$, which exists by the
  minimality of $C$.  Either $U_1$ or $U_2$ must be disjoint from $T$,
  since otherwise Lemma \ref{lem:subset-splayed} would show that $T$ is
  splayed at $p$ and hence along $C$ by $(U_1 \cap T, U_2 \cap T)$.  Therefore
  the one of $U_1, U_2$ that is not disjoint from $T$ contains $T$.

  We proceed by induction on $\#(U \setminus T)$.  In the base case
  $\#(U \setminus T) = 1$, the statement holds, because we know that $U$ is
  splayed and we have just
  shown that one subset in every splaying must contain $T$.

  Now let $(T',U \setminus T')$ be a splaying with $T \subset T'$; fix
  $G \in U \setminus T'$.  Our inductive hypothesis states that
  $U \setminus \{G\}$ is splayed by $(T, U \setminus (T\cup \{G\}))$.
  Let $v \in T_p(X)$: the inductive hypothesis allows us to write $v = t + u$,
  where $t \in \medcap_{D \in T} T_p(D)$
  and $u \in \medcap_{E \in U \setminus (T \cup \{G\})} T_p(E)$.
  Then write $u = q + r$, where
  $q \in \medcap_{D \in T'} T_p(D)$ and $r \in \medcap_{E \in U \setminus T'} T_p(E)$:
  again, we can do this because $(T',U \setminus T')$ is a splaying.
  Of course $v = (t+q) + (u-q)$.  Now $t, q \in \medcap_{D \in T} T_p(D)$, so
  the same holds for $t+q$.  Likewise, $u-q \in \medcap_{D \in T' \setminus T} T_p(D)$
  and $r \in \medcap_{E \in U \setminus T'} T_p(E)$; since these two are equal, they
  are in $\medcap_{D \in U \setminus T} T_p(D)$.  Accordingly every element of the
  tangent space $T_p(X)$ can be written as the sum of an element of
  $\medcap_{D \in T} T_p(D)$ and of $\medcap_{E \in U \setminus T} T_p(E)$.
  This shows that $U$ is splayed by $(T,U \setminus T)$ and completes the
  induction.
\end{proof}

We now return to our problem.

\begin{notation}
  Let $T$ and $Q$ be arbitrary subsets of $S$.  We reorder the elements of
  $S$ so that $T = \{D_1, \dots, D_j\}$ and $Q = \{D_i, \dots, D_k\}$.
  Let $C$ be a connected component of the intersection $\medcap_{n=1}^j D_n$ that
  satisfies the
  same hypotheses as in Proposition \ref{prop:constant-splaying}.  Then we take
  $R$ a connected component of $\medcap Q$.  (As before,
  the connected and irreducible components of the intersection are the same.)
  Let $\pi: \tilde X \to X$ be the blowup along $C$,
  with exceptional divisor $E$,
  and let the $\tilde D$ be the strict transforms.  If $k \le j$, so
  that $C \subseteq R$, then 
  $\pi^{-1}(R) \cap \medcap_{m=i}^k D_i$ is the blowup of $R$
  along $C$ and is smooth (empty, if $i = 1, k = j$).
  If $C \cap R$ is empty,
  then the map $\pi^{-1}(R) \to R$ is an isomorphism, so the
  domain is smooth.  Finally, if $R \subseteq C$, then there is no
  component of $\medcap_{m=i}^k \tilde D_m$ lying above $R$, so there is nothing
  to do.  Since these cases have been dealt with, we assume
  henceforth that $i \le j < k$.

  Thus we take a point $p \in C \cap R$ and a
  point $q$ lying above it in $\tilde X$.  Let $F$ be the fibre above $p$ in
  $\tilde X$.  Finally, let $\tilde T = \medcap_{m=1}^j \tilde D_m$, let
  $\tilde Q = \{\tilde D_i,\dots,\tilde D_k\}$, and
  $\tilde R = \pi^{-1}(R) \cap \medcap_{m=i}^k Q$.  Let $M$ be the
  component of $\tilde R$
  that maps dominantly to $R$ (it is unique because the blowup
  is an isomorphism away from $C$ and $R \not \subseteq C$);
  we refer to $M$ as the {\em main component}.
  We will show in Corollary \ref{cor:irred} that there are no others.
\end{notation}

Since $R$ does not contain $C$, the set $T \cup Q$ is splayed at $p$,
and by Proposition \ref{prop:one-sided}, it is splayed by $(T,Q \setminus T)$.
We use this to control the intersection of the tangent spaces of the
elements of $\tilde Q$.

In particular, let $V_1 = \medcap_{D \in T} T_p(D)$ and
let $V_2 = \medcap_{E \in Q \setminus T} T_p(E)$.
Clearly $V_1 \subseteq T_p(C)$ and
$V_2$ contains a complement to $V_1$ in $T_p(X)$.

\begin{lemma}\label{lemma:tangent-fibre}
  Let $C,C' \subset X$ be smooth subvarieties with smooth intersection
  and let $p \in C \cap C'$.
  Let $\tilde X \to X$ be the blowup along $C$ with exceptional divisor $E$
  and let $\tilde C'$ be the
  strict transform.  Then the fibre above $p$ in $\tilde C'$ has dimension
  $\dim C' - \dim (C \cap C') - 1$.
\end{lemma}

\begin{proof} This is clear if $C' \subseteq C$, since $\tilde C'$ is then
  empty.  Otherwise $\dim \tilde C' = \dim C'$ and $\tilde C' \cap E$ is a
  divisor in $\tilde C'$
  that maps to $C \cap C'$ with fibres of constant dimension.  The result
  follows from this.
\end{proof}

\begin{cor}\label{cor:fibre-divisor}
  If $C'$ is a divisor in $X$, then either $C \subseteq C'$ and
$\tilde C'$ meets each fibre of the blowup $\tilde X \to X$
in a divisor, or $C \not \subseteq C'$ and
$\tilde C'$ contains the fibres at points of $C \cap C'$.
\end{cor}

\begin{notation}
  Let $c = \dim C$.  Let $g$ be the
  dimension of the component of $\medcap_{n=i}^j D_n$ containing $C$ and let
  $h = \dim \left((\medcap_{n=j+1}^k) D_n \cap C \right)$.
\end{notation}

\begin{lemma}\label{lemma:ker-tangent-space}
  $\dim (\ker T_q({\tilde R}) \stackrel{f}{\to} T_p(R)) \le g-c-1$,
  with equality if $q \in M$.
\end{lemma}

\begin{proof} The kernel is the intersection with $T_q(F)$ by
  Lemma \ref{lem:tangent-space-meet}.
  For $n > j$ we have
  $T_q(\tilde D_n) \supseteq T_q(F)$ by Corollary \ref{cor:fibre-divisor},
  so these may be discarded from the
  intersection.  The result now follows from Corollary \ref{cor:fibral-dirs}:
  the fibral tangent directions in $\medcap_{n=i}^j \tilde D_n$ at $q$ correspond to
  the normal directions to $C$ in $f(M)$ at $p$ mod the $1$-dimensional
  subspace giving $q$.  If $q \in M$, we
  obtain all the tangent directions in $\medcap_{n=i}^j T(\tilde D_n)$, which has
  dimension $g$; otherwise we have a subspace of this space.  In either
  case we take the quotient by the intersection with the
  space spanned by the
  $c$-dimensional
  tangent space to $C$ (which is contained in the space for $q \in M$)
  and the direction of $q$, giving the result claimed.
\end{proof}

\begin{lemma}\label{lemma:im-tangent-space}
  $\dim (\im T_q({\tilde R}) \stackrel{f}{\to} T_p(R)) \le h+1$, with equality if $q \in M$.
\end{lemma}

\begin{proof}
  We start by intersecting the domain with $T_q(E)$ and considering the map
  to $T_p(C)$.  This time Corollary \ref{cor:fibre-divisor} shows that
  it is the $n \le j$ that may be ignored, since the
  corresponding $T({\tilde D_n})$ map surjectively to $T_p(C)$.  For the others
  we have $df (T_q(\tilde D_n) \cap T_q(E)) \subseteq T({D_i \cap C})$,
  which has dimension $h$, and equality holds if $q \in M$.
  Now, consider any curve in $\medcap_{n=j+1}^k D_n$ smooth at $p$
  whose first-order tangent vector is $q$
  % \footnote{Colin: this argument
  %  and similar ones later in the paper are inelegant.  Adam: agreed.  If
  %  we can use an argument that is purely infinitesimal and local, that would
  %  be an improvement.  The problem is that we need a second-order jet
  %  contained in all of the $D_n$ that lifts the tangent vector $q$, not just
  %  $q$ itself.}
  (this exists unless the component
  of $\medcap_{n=j+1}^k D_n$ containing $p$ is a point, in which case there is
  no point $q$ lying above $p$ in $\medcap \tilde D_n$): it lifts to the
  intersection of strict transforms to give a tangent vector that is in the
  direction of $q$ up to $T_p(C)$, so
  $\dim \im df - \dim df (T_q(\tilde D_n) \cap T_q(E)) > 0$.
  Further, any such vector together with $T_p(C)$ generates the image of
  $df$, so the difference is at most $1$.
\end{proof}

\begin{cor}\label{cor:dim-tangent-space}
  $\dim T_q({\tilde R}) = g+h-c$ for $q \in M$.
\end{cor}

\begin{remark}\label{rem:easy-equality}
  In fact, equality in the last two lemmas for $q \in M$
  follows from a simpler argument.  The lemmas show that
  $\dim T_q(\tilde R) \le \dim M$.  But $M \subseteq \tilde R$ implies
  that $\dim M \le \dim T_q(\tilde R)$, so this must
  be an equality, and so we have equality in the two lemmas.
\end{remark}

\begin{cor}\label{cor:irred}
  $\tilde R$ is irreducible.
\end{cor}

\begin{proof}
  Every fibre of $f$ restricted to $\tilde R$ is a linear
  subspace of projective space and is therefore connected; since $R$ is
  connected, the same follows for $\tilde R$.  Thus suppose that $\tilde R$ is
  reducible.  If so, then there is a component $N$ that is not
  disjoint from $M$, by connectedness.
  At a point $r \in N \cap M$, the local dimension of $R$ is at least
  $\dim M = g+h-c$; but $r$ is a singular point of $R$, and hence the dimension
  of $T_r(R)$ is greater than this, contradiction.
\end{proof}

\begin{prop}\label{prop:same-dim}
  The dimension of $R$ is also $g+h-c$.
\end{prop}

\begin{proof}
  Indeed, we showed in Proposition \ref{prop:one-sided} that $\{D_1,\dots,D_k\}$
is splayed at $p$ by $(\{D_1,\dots,D_j\},\{D_{j+1},\dots,D_k\})$.
It follows that either $i = j+1$ or $Q$ is splayed by
$(\{D_i,\dots,D_j\},\{D_{j+1},\dots,D_k\})$.

We now recall that a splaying means that the sum of the intersections of
tangent spaces of the divisors on the two sides at a point spans the tangent
space and use the fact that $\dim V + \dim W = \dim (V+W) + \dim V \cap W$ for
subspaces of a vector space.  From the first statement, we conclude that
$$\dim \bigcap_{n=j+1}^k T_p({D_n}) = \dim X + \dim \bigcap_{n=1}^k T_p({D_n}) - \dim \bigcap_{n=1}^j T_p({D_n}) = \dim X + h - c.$$
From the second, we see that
\begin{align*}
  \dim \bigcap_{n=i}^k T_p({D_n}) & = \dim \bigcap_{n=i}^j T_p({D_n}) + \dim \bigcap_{n=j+1}^k T_p({D_n}) - \dim X \\
  & = g + \dim X + h - c - \dim X = g +h - c.
\end{align*}

\end{proof}

\begin{cor}\label{cor:blowup-smooth}
  $\tilde R$ is smooth.
\end{cor}

\begin{proof}
  By the smoothness of $\medcap_{n=i}^k \tilde D_n$
  and the fact that $\dim R = \dim \tilde R$ we have
  $$\dim \bigcap_{n=i}^k \tilde D_n = \dim \bigcap_{n=i}^k D_n = \dim \bigcap T_p({D_n}) = \dim \bigcap T_q({\tilde D_n}).$$
\end{proof}

At this point we are ready to complete the proof of Theorem \ref{thm:main}.
%\footnote{Adam: this used to say that we had just done so, which wasn't quite true.}
After that, we will establish our resolution procedure
in Proposition \ref{prop:resolution}.

\begin{prop}\label{prop:blowup-still-splayed}
  Let $B$ be a component of the intersection of some $D_i$ that is splayed as
  $(S_1,S_2)$.  Then the corresponding component of the intersection of the
  same $\tilde D_i$ is splayed as $(\tilde S_1,\tilde S_2)$, that is, by the
  same decomposition of the set of indices.  In addition, the union of any
  nonempty
  subset $U \subseteq \{\tilde D_i\}$ with $\{E\}$ is splayed by $(\{E\},U)$.
\end{prop}

\begin{proof}
  For the first statement, note that since all intersections are smooth and
  of the same dimension both before and after the blowup, we have
  $\dim V = \dim \tilde V, \dim W = \dim \tilde W, \dim (V \cap W) = \dim (\tilde V \cap \tilde W)$, where $V, W$ are the appropriate components of the
  intersections of the two sets in the splaying and $\tilde V, \tilde W$ those
  of the intersections of the strict transforms.  Since
  $\dim V + \dim W - \dim (V \cap W) = n$, the same holds with tildes
  everywhere.  

  For the second statement, it suffices to find a vector tangent to all the
  $\tilde D_i \in U$ at a point $q$ lying above $p$
  that is not contained in $E$.  As before, we do this by choosing
  a curve through $q$ that lifts a curve in $\medcap D_i$ smooth at $p$ and whose
  tangent vector is not in the direction of $C$.  If $p$ is an isolated point
  of the intersection $\medcap D_i$, then there is no $q$ and nothing to do.
  (Note that our assumptions require that more than one $D_i$ pass through
  the point $p$.  Otherwise, the codimension of $p$ would be $1$ and only
  one divisor could pass through it; thus there would be no splaying.)
  This proves Theorem~\ref{thm:main} (1): the strict transform of the
  non-splayed component that was blown up is no longer a component of the
  intersection and no splayed component has become non-splayed.
  It also proves Theorem~\ref{thm:main} (2).
\end{proof}

We complete the proof of Theorem \ref{thm:main} by proving the third
statement, that $\tilde S \cup \{E\}$ is an arrangement.

\begin{proof} As we have already shown in Corollary \ref{cor:blowup-smooth}
  that all components of intersections of $\tilde S$ are smooth, it is
  enough to consider the intersections of these with $E$.  We just showed that
  $T_q(\tilde R) \not \subseteq T_q(E)$, where $\tilde R$ is such a component;
  it follows that $\dim (T_q(\tilde R) \cap T_q(E)) = \dim T_q(\tilde R) - 1$.
  Since $\tilde R \not \subseteq E$, we also have
  $\dim (\tilde R \cap E) = \dim \tilde R - 1$, and, in light of the smoothness
  of $R$, that $R \cap E$ is smooth as well.
\end{proof}
  
The next two lemmas and corollary are standard results for which no
originality is claimed; we include a proof for the reader's convenience.
The next lemma shows that a normal crossings divisor is locally
a Boolean arrangement of hyperplanes in the sense of
\cite[Example 1.8]{orlik-terao}.

\begin{lemma}\label{lem:completely-transverse} Let $D_1, \dots, D_m$ be
  an arrangement of divisors.  Then $\medcup_{i=1}^m D_i$ is a normal
  crossings divisor on $X$ if and only if all components of the intersection
  of two or more $D_i$ are splayed.
\end{lemma}

\begin{proof} First suppose that $\medcup_{i=1}^m D_i$ is a normal crossings
  divisor and let $C$ be a component of the intersection of
  $D_{n_1}, \dots, D_{n_k}$.  Then locally near a point of $C$ the $D_{n_i}$
  can be given by $x_i = 0$.  Thus any partition into two nonempty subsets
  constitutes a splaying.

  To prove the converse, we show by induction on $k$ that the union of any
  $k$ of the $D_i$ is a normal crossings divisor.
  First, if we consider a single $D_i$, it is
  certainly a normal crossings divisor since it is smooth.
  Now suppose the result known for sets of fewer than $k$ divisors, and
  consider a component of the intersection $\medcap_{i=1}^k D_{n_i}$.
  Suppose that $\{D_{n_i}\}$ is splayed by
  $E \cup F$.  Then $\medcup E$, $\medcup F$ are normal crossings divisors by
  inductive hypothesis, so each one is locally defined by a product of
  linearly independent linear forms.  But these forms can be expressed in
  disjoint sets of variables, by the hypothesis that $\{D_{n_i}\}$ is splayed by
  $E \cup F$.  It follows that $\{D_{n_i}\}$ is expressed locally along $C$
  as a product of distinct coordinates.
\end{proof}
  
\begin{lemma}\label{lemma:all-splayed-easy} Let $\{D_1, \dots, D_m\}$ be
  an arrangement of divisors on $X$ such that $\medcup_{i=1}^m D_i$ is a normal
  crossings divisor as in Remark \ref{lem:completely-transverse}.
  Then every component of the intersection of $r$ divisors has codimension~$r$
  (the empty scheme is considered as having no components).
  Further, every such component is splayed by every nontrivial partition of
  the set of divisors containing it.  Conversely, if every component of the
  intersection of $r$ divisors in an arrangement has codimension $r$, then
  the union of the arrangement is a normal crossings divisor.
\end{lemma}

\begin{proof} For the first statement, consider
  a counterexample with minimal $r$: clearly $r > 1$.
  Suppose that we have the splaying $(\{D_1,\dots,D_q\},\{D_{q+1},\dots,D_r\})$.
  Then $\dim \medcap_{i=1}^q D_i = \dim X - q$ and
  $\dim \medcap_{i=q+1}^r D_i = \dim X - (r-q)$.  Since the intersections are smooth
  and the sum of the tangent spaces at a point $p$ is $T_p(X)$, the intersection
  of the tangent spaces has dimension $\dim X - r$, so the component has
  dimension $r$, contradiction.

  For the second statement, it follows from the first that, for every point
  $p$ on an intersection component $C$ of $D_1, \dots, D_k$,
  the tangent spaces of the divisors
  containing it are subspaces of codimension $1$ of $T_p(X)$ in general
  position.  It follows immediately that the intersection of
  $s$ of them is of codimension $s$.  Applying this to a proper nonempty
  subset $T \subset S = \{D_1,\dots,D_k\}$, its complement, and $S$, we see that
  the intersections of tangent spaces over the subset and its complement
  have codimension $\#T, k-\#T$, while their intersection has codimension $k$.
  Thus the sum of the intersections is $T_p(X)$ as desired.

  The converse again follows from the statement that
  $\dim (V \cap W) + \dim (V + W) = \dim V + \dim W$.
\end{proof}

\begin{cor}\label{cor:blowup-completely-splayed}
  Let $\{D_1,\dots,D_m\}$ be an arrangement of divisors on
  $X$ whose union is a normal crossings divisor
  and let $C$ be a component of the intersection of a subset of the
  $D_i$.  Then the union of the divisors in the
  arrangement $\{\tilde D_1,\dots,\tilde D_m\} \cup \{E\}$
  on the blowup $\tilde X$ of $X$ along $C$ is also a normal crossings divisor.
\end{cor}

\begin{proof}
  For subsets of $\{\tilde D_1,\dots,\tilde D_m\}$ the statement follows
  immediately from Proposition \ref{prop:blowup-still-splayed}.
  By Lemma \ref{lemma:all-splayed-easy}, it suffices to show that
  $\medcap_{i=1}^m T_q(\tilde D_i)$ is not contained in $T_q(E)$ for any
  $q \in \medcap \tilde D_i$.  Let $C$ be the component of $\medcap \tilde D_i$
  containing $q$; if $\dim C = 0$, then $C$ pulls back to an empty scheme
  and the statement is vacuous.  Otherwise, take a curve in $\medcap D_i$
  in the direction of $q$; the tangent vector of its pullback to $\tilde X$
  is not in $T_q(E)$.
\end{proof}

\begin{prop}\label{prop:resolution} Suppose given an arrangement of divisors
  for which every intersection component is either admissible or splayed and
  a double cover $Y \to X$ whose branch locus is the union of the divisors in
  the arrangement.  Then $Y$ admits a crepant resolution of singularities.
\end{prop}

\begin{proof}
  Following the procedure in \cite[Proposition 5.6]{cynk-hulek},
  we inductively blow up minimal non-splayed intersections.  By hypothesis,
  these are admissible, so by Proposition \ref{prop:admissible-is-crepant}
  they are crepant blowups.  By
  Corollary \ref{cor:blowup-smooth} and Proposition
  \ref{prop:blowup-still-splayed}, this preserves the inductive hypothesis
  that our divisors constitute an arrangement in which all intersection
  components are admissible or splayed; further, the number of non-splayed
  intersection components decreases at every step.  Thus after finitely
  many steps we obtain an arrangement whose union is a normal crossings
  divisor.  From there,
  we successively blow up all nonempty intersections of two divisors in the
  branch locus.  Again, this preserves the status of the union of the
  arrangement being a normal crossings divisor,
  and each of these is an admissible step and does not
  require introducing the exceptional divisor into the branch locus; thus after
  finitely many steps the divisors are pairwise disjoint.  At that point
  we take the fibre product of $Y \to X$ with the composition of the blowups
  to obtain the desired resolution.
\end{proof}

\begin{remark}\label{ref:uniqueness-derived-equivalence}
  The resolution of $Y$ constructed by the procedure of Proposition
  \ref{prop:resolution} is not unique, because blowing up a subvariety
  $C$ and then the strict transform of $C'$ is not the same as blowing up
  $C'$ and then the strict transform of $C$, unless $C \cap C'$ is empty.
  Even if we start with a branch locus that is a normal crossings divisor, so
  that all intersections are splayed as in Lemma \ref{lemma:all-splayed-easy},
  the resolution depends on the order in which we blow up the intersections
  of two of the components.  The combinatorial problem of expressing the
  number of different resolutions in terms of the sets of components with
  nonempty intersection seems to be difficult and will not be treated here.

  It is natural to ask about the relation between different resolutions
  constructed by our procedure.  Since they are constructed by blowing up
  the same loci or their strict transforms, they are isomorphic in
  codimension $1$, and hence the pullbacks of their canonical divisors to
  the fibre product over $Y$ are linearly equivalent: thus
  they are $K$-equivalent in the sense of
  \cite[Definition 1.1]{kawamata}.  Conjecture 1.2 of \cite{kawamata}
  then predicts that they are derived equivalent.  Can this be proved directly?
  %\footnote{Adam: Colin, please look at what I've written here.}
\end{remark}

\section{Primary decomposition of the singular subscheme of an arrangement}\label{sec:primary-dec}
In this section, all of our constructions are local, so we need only consider
affine schemes, and we are not concerned with crepant resolutions, so the
condition of admissibility is not relevant.  By the {\em singular subscheme}
of an affine scheme we mean the subscheme defined by the Jacobian ideal.
%Let $\{D_1,\ldots,D_n\}$ be an arrangement of divisors on $X$.
In this section we prove the following theorem:
\begin{thm}\label{thm:pc}
  Let $\{D_i\}$ be an arrangement of divisors on $X$,
  and let $S$ be the singular subscheme of $\medcup_i D_i$.
  Let $E = \medcap_{i \in S} D_i$ be an intersection of
  some of the $D_i$ which is splayed and not of codimension $2$.
  Then the minimal primary decomposition of
  $S$ does not contain a component supported on $E$.
\end{thm}

In light of the resolution procedure described in
Proposition~\ref{prop:resolution},
this statement has the following intuitive explanation.  Let $E$ be as in the
theorem.  By repeatedly blowing up the primary components of $S$ and their
strict transforms, we make the $D_i$ pairwise disjoint and hence $S$ empty.
Since this can be done without ever blowing up a strict transform of $E$,
it must not have been a genuine component of~$S$.

The theorem can be rephrased as follows: if $\cod E > 2$ and $E$ is supported
on a primary component of
$S$, then $E$ is not splayed.  We pose the converse as the following question.
\begin{question}
  Let $\{D_1,\ldots,D_n\}$ be an arrangement of divisors on $X.$
  Let $E = \medcap D_i$, and suppose that $\cod E >2$ and $E$ is not
  supported on a primary component of $S$.  Does this imply that $E$ is splayed?
\end{question}
A positive answer to this question would allow one to run
Algorithm~\ref{alg:resolve} by computing the primary decomposition of the
singular subscheme of $\medcup D_i$ and repeatedly blowing up the minimal
components whose support does not appear (which correspond to the ideals
maximal among the ideals of these components).

Recall that a primary decomposition of an ideal
$$I = \q_1\cap \cdots \cap \q_k$$
is minimal if the radicals $\sqrt{\q_i}$ are distinct and the collection of primary ideals is {\it irredundant}, i.e. for any $j$ we have
$$\q_j \not\supset \bigcap_{\substack{i = 1 \\ i \neq j}}^k \q_i.$$

We first note the following lemma.
\begin{lemma}[{\cite[Exercise 4.7 (e)]{a-m}}]
  Let $I$ be an ideal in a commutative ring $A$.
  Let $I = \q_1\cap \cdots \cap \q_n$ be a primary decomposition.
  Then $I[t] =\q_1[t] \cap \cdots \cap \q_n[t]$ is a primary decomposition in $A[t]$.
\end{lemma}

\begin{defn}\label{def:strong-loc}
  Let $X$ be a variety and $p$ a point of $X$.  A {\em strong localization}
  of the local ring $\O_{X,p}$ at $p$ is either its henselization or its
  completion (all arguments will apply equally to both).
\end{defn}

Write $\overline{x}$ for $x_1,\ldots,x_n$ and $\overline{y}$ for $y_1,\ldots,y_m$. Let $S = k[\overline{x},\overline{y}]$.  Let
$f_1,\ldots,f_N \in k[\overline{x}]$ and $g_1,\ldots,g_M \in k[\overline{y}]$
be irreducible polynomials which generate distinct principal ideals.
Let $F = \prod f_i$ and $G =\prod g_j$.  Note that
$\medcap_i V(f_i) \cap \medcap_j V(g_j)$ is splayed by
$$(\{V(f_1),\ldots,V(f_N)\},\{V(g_1),\ldots,V(g_M)\}$$ and any splaying will
have a similar form for a strong localization.

Given a polynomial $f \in S$ we write $$J_f = \left(\frac{\partial f}{\partial x_1},\ldots,\frac{\partial f}{\partial x_n},\frac{\partial f}{\partial y_1},\ldots,\frac{\partial f}{\partial y_m}\right)$$ for the Jacobian ideal.
In \cite{faber}, the following statement 
is proved in the context of germs of holomorphic functions, but the proof
applies equally well to polynomial rings or Henselian local rings, and in
particular to complete local rings.%\footnote{Adam: I reworded this a bit;
  %it seems unnecessary to refer separately to complete and to Henselian
  %local rings.}
\begin{lemma}\cite[Prop.~3.5]{faber}
  For polynomials $F\in k[\overline{x}]$ and $G\in k[\overline{y}]$ in $S$ we have the decomposition
  $$(J_{FG},FG) = (F,G) \cap (J_F,F) \cap (J_G,G).$$
\end{lemma}

\begin{thm}\label{thm:decomp-disjoint} Let $k$ be an algebraically closed field.  Suppose that  $F = \prod f_i \in k[\overline{x}]$ and $G = \prod g_j \in k[\overline{y}]$, where the $f_i$, $g_j$ are
  irreducible.
  %\footnote{Colin: the point is that this is a splaying, at least locally.}
  Let $(J_F,F) = \medcap \p_i$ and $(J_G,G) = \medcap \q_j$ be primary decompositions in the rings $k[\overline{x}],k[\overline{y}]$ respectively. 
  Then
  $$(J_{FG},FG) = \bigcap_{i,j} (f_i,g_j) \cap  \bigcap \p_i[\overline{y}] \cap  \bigcap \q_j[\overline{x}]$$
is a minimal primary decomposition.
\end{thm}
\begin{proof}
  Since the tensor product of integral domains over an algebraically closed
  field is an integral domain, which follows from \cite[Corollary 2 to Theorem IV.24]{jacobson},
  and $S/(f_i,g_j) \simeq k[\overline{x}]/(f_i) \otimes k[\overline{y}]/(g_j)$,
  we conclude that $(f_i,g_j)$ is a prime ideal.
  With the above lemmas, we only need to check that the decomposition is
  irredundant and that the radicals are distinct.

  First we show that $\q_d[\overline{x}]$ cannot be removed.
  Since the $\q_i$ are irredundant, there is
  $a \in \medcap \q_i \setminus \q_d \subseteq k[\overline{y}]$.
  Consider the element $Fa$.  Note that $F \in (f_i,g_j)$ for all pairs
  $i,j$ and $F \in (J_F,F) = \medcap \p_i$.  So we see that
  $$Fa \in \bigcap (f_i,g_j) \cap \bigcap \p_i[\overline{y}] \cap \bigcap_{j \neq d} \q_d[\overline{x}].$$
  If $Fa \in \q_d[\overline{x}]$, then
  $F \in \sqrt{\q_d[\overline{x}]} = \sqrt{\q_d}[\overline{x}]$ because
  $a \notin \q_d[\overline{x}]$.
  So $F \in k[\overline{x}] \cap  \sqrt{\q_d}[\overline{x}] =  0$.
  The same argument shows that we cannot remove $\p_i[\overline{y}]$.
  
  Now fix a particular $k,\ell$.  We need to show that 
  $$(f_k,g_\ell) \not\supseteq  \bigcap_{(i,j) \neq (k,\ell)} (f_i,g_j) \cap  \bigcap \p_i[\overline{y}] \cap  \bigcap \q_j[\overline{x}].$$
  First note that $\cod_{k[\overline{x}]} (f_k) =1$ and $\cod_{k[\overline{x}]} (J_F,F) = \cod_{k[\overline{x}]} \medcap \p_i \geq 2$.
  So $V(\medcap \p_i) \not\supseteq V(f_k)$ and so $\medcap \p_i \subseteq \sqrt{\medcap \p_i} \not\subseteq \sqrt{(f_k)} = (f_k)$.  It follows that there is
  $a \in \medcap \p_i \setminus (f_k) \subset k[\overline{x}]$.
  Similarly, there is an element $b \in \medcap \q_i \setminus (g_\ell) \subset k[\overline{y}]$.  Now consider the element
  $$z = \frac{FGab}{f_kg_\ell} \in \bigcap_{i,j \neq k,\ell} (f_i,g_j) \cap \bigcap \p_i \cap \bigcap \q_j.$$
  Suppose that $z \in (f_i,g_j)$.  Since this ideal is prime and $FG/f_kg_\ell$
  is not in $(f_i,g_j)$, we can conclude that $a$ or $b$ is in $(f_i,g_j)$.
  If $a \in (f_i,g_j)$, we have $a \in (f_i,g_j) \cap k[\overline{x}] =(f_i)$ giving a contradiction, and similarly for $b \in (f_i,g_j)$.

  Lastly we show the radicals are distinct.  
  We note that $\sqrt{\p_i[\overline{x}]} = \sqrt{\p_i}[\overline{x}]$. So these ideals are preserved by pushing forward and pulling back:
  $$I \mapsto (k[\overline{x}]\cap I)k[\overline{x},\overline{y}],$$
  whereas this does not hold for the ideals $\q_i$ and $(f_i,g_j)$.  So the
  radicals $\sqrt{p_i}$ are distinct from $\sqrt{q_j},(f_i,g_j)$.  We argue similarly for the $\sqrt{q_i}$ to complete the proof. 
  \end{proof}
A version of the above result for strong localizations follows from the next two lemmas, which are well-known results with standard references.
\begin{lemma} Let $m=(\overline{x},\overline{y})$ and
  $S_{xy}' = k[\overline{x},\overline{y}]_m$ be the localization at the origin.
  We let $S_{xy}$ denote the strong localization of $S$.
  Starting from $S'_x =k[\overline{x}]$
  and $S'_y = k[\overline{y}]$, we similarly define $S_x,S_y$.  Let
  $i_x:S_x \to S_{xy},$ and $i_y : S_y \to S_{xy}$ be the obvious injections.
  Let $\mathfrak{q}$ be a primary ideal in $S_x$.  Then:
  \begin{enumerate}
  \item $i_x(\mathfrak{q})S_{xy}$ is $i_x(\sqrt{\mathfrak{q}})S_{xy}$-primary in $S$;
  \item $\sqrt{i_x(\mathfrak{q})S_{xy}} = i_x(\sqrt{\mathfrak{q}})S_{xy}$;
  \item and the operation $\mathfrak{q} \mapsto i_x(\mathfrak{q})S_{xy}$ preserves minimal primary decompositions.
  \end{enumerate}
\end{lemma}
\begin{proof}
  All three items follow from~\cite[IV, \S 2.6, Prop.~11]{bour-comalg}
  and the fact that the composition $S_x \to S_{x}[\overline{y}]_{m} \to S_{xy}$ is faithfully flat.
\end{proof}
  \begin{lemma}
  Let $U,U'$ be smooth affine open varieties
  over an algebraically closed field $k$.  Let $f\in \O(U)$ and $g\in
  \O(U')$ define normal hypersurfaces.  Let $p\in V(f)$ and $q \in V(g)$.
  Then the strong localization of $V(f,g) \subset U\times U'$
  at the point $(p,q)$ is a normal domain.
  \end{lemma}
\begin{proof}
  Since $\O(U)/(f)$ and $\O(U')/(g)$ are noetherian integrally closed domains,
  it follows that
  $\O(V(f,g)) = \O(V(f)) \times \O(V(g))$ is an integrally closed domain
  by~\cite[2,6.14.1]{EGAIV} and~\cite[Corollary 2 to Theorem IV.24]{jacobson}. 
  By Zariski's main theorem for power series~\cite[\S 9, III]{MumfordRed},~\cite[vol.~2, pp.~313-320]{ZariskiSamuel}, the completion at $(p,q)$ is
  an integrally closed domain.  
  The same is true for the henselization by~\cite[Ch.~VIII, \S 4, Th.~3]{RaynaudHensel}.
\end{proof}
Applying these lemmas to the proof of Theorem~\ref{thm:decomp-disjoint} gives the following corollary.
\begin{cor}
  Let $X$ be a smooth variety over an algebraically closed field.
  Let $p \in X$ be a point and let $\{D_1,\ldots,D_N,D_1',\ldots,D_M'\}$
  be a set of divisors in $X$ with $p \in D_i,D_j'$ for all $i$ and $D_i,D_j'$
  locally normal at $p$.  Write $D_i = V(f_i)$ and $D_j'=V(g_j)$ for
  $f_i,g_j$ in a strong localization of $\O_{X,p}$.
  Suppose further that a set $\{x_1,\ldots,x_n,y_1,\ldots,y_m\}$ of generators
  $\m/\m^2$, where $\m$ is the maximal ideal of the strong localization,
  can be chosen such that $f_i$
  is expressed in terms of the $x_i$ and the $g_j$ are expressed in terms of
  the $y_j$.  Then the primary components of the singular subscheme of
  $\medcup_i D_i\cup \medcup D_j'$ are the $V(f_i,g_j)$ and the primary
  components of the singular subschemes $\medcup D_i$ and $\medcup D_j'$.
  \end{cor}
  
We now complete the proof of Theorem \ref{thm:pc}.

\begin{proof}
  The statement is purely local and we are assuming that the $D_i$ are smooth
  and pairwise transverse.  Let $p \in E$: we start by passing to the strong
  localization at $p$, so that we may assume that $X = \A^n$ and that the
  $D_i$ are hyperplanes.
  Then all intersections of
  $D_i$ are linear subspaces of $\A^n$; in particular, they are irreducible.

  Suppose that $E$ is splayed by $(E_1,E_2)$.  We may choose coordinates
  $x_{1,j}, x_{2,k}$ such that the elements of $E_i$ are expressed in terms of
  the $x_{i,*}$.  By definition the $E_i$ are nonempty, so there is at least
  one coordinate of each type.  Then, by Theorem \ref{thm:decomp-disjoint},
  there is a primary decomposition of $S$ supported on the $(x_{1,j}, x_{2,k})$
  and on ideals generated by elements supported only on the $x_{1,j}$ or the
  $x_{2,k}$.  Only the first type of this can coincide with $E$, but if so then
  $E$ has codimension $2$.
\end{proof}

{\em Acknowledgments.}  We would like to thank Eleonore Faber for some very
helpful discussions of the concept of a splayed set of divisors, and
S. Cynk and K. Hulek for clarifying a point in their paper \cite{cynk-hulek}.
C.~Ingalls was partially supported by an NSERC Discovery Grant.
A.~Logan would like to thank the Tutte Institute for Mathematics and Computation
for its support of his external research.
  
\bibliography{forCHpaper}{}
\end{document}